\documentclass[12pt]{article}
\usepackage{ulem}
\usepackage{forest}
\usepackage[colorlinks, citecolor=red]{hyperref}
\usepackage{graphicx,amsmath,amsfonts,amssymb,color,mathrsfs, amsmath,amsthm}
\usepackage{algorithm}
\usepackage{algpseudocode}
\allowdisplaybreaks[4]
\usepackage{xltabular}

\UseRawInputEncoding
\usepackage{epsfig}
\bibdata
\bibstyle

\newcommand{\chuhao}{\fontsize{19pt}{\baselineskip}\selectfont}

\numberwithin{equation}{section}

 \newtheorem{theorem}{Theorem}[section]
 \newtheorem{lemma}{Lemma}[section]
 \newtheorem{assumption}{Assumption}[section]
 \newtheorem{proposition}{Proposition}[section]
 
 \newtheorem{remark}{Remark}[section]

 \newtheorem{example}{Example}[section]

 \newtheorem{definition}{Definition}[section]

\title{\bf\color{black} \chuhao{Perturbation Analysis of Markov Chain Monte Carlo for Graphical Models}}
\date{}

\begin{document}
\author{
Na Lin \thanks{School of Mathematics and Statistics, HNP-LAMA, Central South University, China.}
\and Yuanyuan Liu \footnotemark[1]
\and Aaron Smith \thanks{Corresponding author. Department of Mathematics and Statistics, University of Ottawa, Canada.}
}

\maketitle
\begin{abstract}
The basic question in perturbation analysis of Markov chains is: how do small changes in the \textit{transition kernels} of Markov chains translate to chains in their \textit{stationary distributions}? Many papers on the subject have shown, roughly, that the change in stationary distribution is small as long as the change in the kernel is much less than some measure of the convergence rate. This result is essentially sharp for generic Markov chains. In this paper we show that much larger errors, up to size roughly the \textit{square root} of the convergence rate, are permissible for many target distributions associated with graphical models. The main motivation for this work comes from computational statistics, where there is often a tradeoff between the \textit{per-step error} and \textit{per-step cost} of approximate MCMC algorithms. Our results show that larger perturbations (and thus less-expensive chains) still give results with small error.


\end{abstract}

\section{Introduction}

Informally, perturbation bounds for Markov chains look like the following: if the distance $d(Q,K)$ between two transition kernels $Q,K$ is sufficiently small, then the distance $d(\mu,\nu)$ between their stationary measures $\mu, \nu$ is also small. Many results of this form, such as Corollary 3.2 of \cite{mitrophanov05} and followup work \cite{johndrow2017error,rudolf18,negrea_rosenthal_2021}, require that the inverse error  $d(Q,K)^{-1}$ be much larger than some notion of the ``time to convergence" $\tau(Q)$ of one of the two chains. The main results of this paper show that, in the special case that $\mu, \nu$ both correspond to graphical models and $Q,K$ ``respect" the same graphical model in a sense made precise in this paper, we can ensure that $d(\mu,\nu)$ remains small even for much larger errors $d(Q,K)$. Our main result, Theorem \ref{mainThm},  allows errors of up to size $d(Q,K) \approx \tau(Q)^{-1/2} \gg \tau(Q)^{-1}$. Our main illustrative example, in Section \ref{minibatch_sec}, shows how these bounds can be achieved in a simple setting. We also note by example that both the existing bounds and our new bounds are essentially sharp for certain large and natural classes of Markov chains; see Examples \ref{ExSharpnessSimple}, \ref{ExSharpnessSimple-tv} and \ref{ExSharpnessProduct}.

The main motivation for our work is the analysis of ``approximate" or ``noisy" Markov chain Monte Carlo (MCMC) algorithms. In this setting, we think of $Q$ as an ideal Markov chain that we would like to run, but which is computationally expensive. We think of $K$ as an approximation to $Q$ that is less computationally expensive. As a prototypical example inspired by \cite{Korattikara2013AusterityIM}, $Q$ might be the usual Metropolis-Hastings algorithm targeting the posterior distribution associated with a large dataset of size $n$, while $K = K_{m}$ might be an algorithm that approximates the posterior distribution using a data subsample of size $m \ll n$. We expect the per-step cost of $K_{m}$ to increase with $m$, while we also expect the error $d(Q,K_{m})$ to decrease with $m$. This suggests a tradeoff: we would like to choose a value of $m$ that is large enough for the stationary distributions of $Q,K_{m}$ to be close, but small enough for $K_{m}$ to be computationally tractable. Improved perturbation bounds let us confidently decrease the value of $m$ while still obtaining good results.

\subsection{Relationship to Previous Work}

There is a long history of analyzing perturbations of matrices and other linear operators - see e.g. the classic textbook \cite{Kato1995}. Many of these results include Markov chains as special cases. The first application of these ideas to MCMC appeared in \cite{BREYER2001123}, which used perturbation bounds to show that ``rounding off" numbers to the typical precision of floating-point numbers will usually result in negligible error.

A recent surge of work on applying perturbation bounds to MCMC was inspired by two important new lines of algorithm development. The first was the development of algorithms, such as the Unadjusted Langevin Algorithm (ULA) of \cite{PARISI1981378}, that are based on discretizations of ``ideal" continuous-time processes. The second was the development of algorithms, such as stochastic gradient Langevin dynamics (SGLD) of \cite{wellingSGLD11} and the ``minibatch" Metropolis-Hastings of \cite{Korattikara2013AusterityIM}, that try to make individual MCMC steps cheaper by using computationally cheap plug-in estimates of quantities appearing in ``ideal" MCMC algorithms. Often, as in both \cite{wellingSGLD11} and  \cite{Korattikara2013AusterityIM}, the plug-in estimates are obtained by using a small subsample of the original dataset. In both lines, the result was a profusion of useful algorithms that can be viewed as small perturbations of algorithms that have very good theoretical support but are computationally intractable.

We focus now on the second line of algorithms. Perturbation theory has been used to try to show that the new computationally-tractable algorithms would inherit many of the good properties of the original computationally-intractable algorithms. Some representative theoretical papers on this subject include \cite{johndrow2017error,rudolf18,negrea_rosenthal_2021}, all of which give generic perturbation bounds that are widely applicable to approximate MCMC algorithms such as \cite{wellingSGLD11,Korattikara2013AusterityIM} under very broad conditions. This generic work has also been applied to more complex algorithms, such as \cite{Horseshoe2020,rastelli2023computationally}.

One of the main questions raised by this work is: how good do the plug-in estimates need to be in order to obtain an algorithm with ``small" error? Theoretical work such as \cite{johndrow2017error,rudolf18,negrea_rosenthal_2021} often required that the distance between the ``original" kernel $Q$ and ``approximate" kernel $K$ satisfy a condition along the lines of
\begin{equation} \label{IneqGenCond}
d(Q,K) \ll \frac{1}{\tau(Q)},
\end{equation}
where $d(Q,K)$ is a notion of distance between kernels and $\tau(Q)$ is a notion of time to converge to stationarity. Study of specific algorithms supported the idea that Inequality \eqref{IneqGenCond} was often necessary in practice and was \textit{not} satisfied by various naive plug-in estimators \cite{bardTall17,Nagapetyan2017,johndrow2020free}. Since then, a great deal of applied work in the area has focused on developing plug-in estimators that do satisfy this inequality \cite{QuirSubSpeed19,CVSGMCMC19}.

This work looks at the question from the opposite point of view. Rather than trying to improve control variates to obtain algorithms that satisfy Inequality \eqref{IneqGenCond}, we try to find conditions that are weaker than  Inequality \eqref{IneqGenCond} and hold under conditions of interest in statistics. \textbf{The main lesson} in the present paper is that generic perturbation bounds can be \textit{vastly} strengthened by restricting attention to specific statistically-interesting settings.

Of course, the point of view in this paper and the point of view in applied work such as \cite{QuirSubSpeed19,CVSGMCMC19} do not compete with each other. Improving control variates and improving the condition both allow you to run MCMC algorithms with smaller per-step costs, and these improvements can be combined. \textbf{The main question left open} in this work is: how easy is it to engineer Markov chains for which conditions such as \eqref{IneqGenCond} can be weakened?

\subsection{Paper Guide}

Section \ref{notationsec} contains the basic notation used throughout this paper. Section \ref{SecSimpleEx} is pedagogical, giving simple examples that aim to illustrate when previous results are or are not sharp. Section \ref{mainresult} states and proves the main result, Theorem \ref{mainThm}. Finally, Section \ref{minibatch_sec} gives a concrete algorithm and model to which Theorem \ref{mainThm} can be applied.

We note here one question that is raised rather early in the paper but not resolved until near the end. As discussed in Remark \ref{RemCheckCond}, there is one main obstacle to constructing an approximate algorithm to which Theorem \ref{mainThm} can be applied: the algorithm's stationary distribution must ``respect" the same conditional independence properties as the original target distribution. The main purpose of Section \ref{minibatch_sec} is to show that, while this property is not automatic, it is not terribly difficult to construct an algorithm with this property. In particular, the pseudomarginal framework introduced in \cite{pseudo09} provides a powerful general tool for doing this.

\section{Notation}\label{notationsec}

We introduce notation used throughout the paper.

\subsection{Generic Probability Notation} \label{SecNotation}

For two probability distributions $\mu, \nu$ on the same Polish space $\Omega$ with Borel $\sigma-$algebra $(\Omega,\mathcal{F})$, define the total variation (TV) distance
\[
d_{TV}(\mu, \nu) = \sup_{A \in \mathcal{F}}|\mu(A) - \nu(A)|.
\]

When $\mu, \nu$ have densities $f,g$ with respect to a single reference measure $\lambda$, the Hellinger distance between $\mu$ and $\nu$ is defined as
\[
d_{H}(\mu,\nu)=\left(\int_{x \in \Omega}\left(\sqrt{f(x)}-\sqrt{g(x)}\right)^2 \lambda(dx) \right)^{1/2},
\]
and the $L^2(\lambda)$ distance is defines as
\[
d_{L^2(\lambda)}(\mu,\nu)=\left(\int_{x\in \Omega}(f(x)-g(x))^2\lambda(dx)\right)^{1/2}.
\]

Our argument will rely on the following well-known relationships  (see, e.g., page 135 in \cite{daskalakis20} and Lemma 20.10 in \cite{Levin2008MarkovCA}):
\begin{equation}\label{tv_h}
\frac{1}{2}d_{H}^2(\mu,\nu)\leq d_{TV}(\mu,\nu)\leq d_{H}(\mu,\nu) \leq d_{L^2}(\mu,\nu).
\end{equation}

By a small abuse of notation, when $d$ is a distance on measures, we extend it to a distance on transition kernels via the formula:

\[
d(Q,K) = \sup_{x \in \Omega} d(Q(x,\cdot), K(x,\cdot)).
\]

Let $Q$ be a transition kernel with stationary measure $\pi$ and let $d$ be a metric on probability measures. Define the associated mixing time:
\[
\tau_{mix}(Q,d,\epsilon) = \min\{ t \, : \, \sup_{x \in \Omega} d(Q^{t}(x,\cdot),\pi) < \epsilon\}.
\]
By convention, we write $\tau_{mix}(Q) = \tau_{mix}(Q,d_{TV},\frac{1}{4}).$

When we have two functions $f,g$ on the same state space $\Omega$, we write $f \lesssim g$ as short hand for: there exists a constant $0 < C < \infty$ such that $f(x) \leq C g(x)$ for all $x \in \Omega$. Similarly, we write $g \gtrsim f$ if $f \lesssim g$, and we write $f \approx g$ if both $f \lesssim g$ and $g \lesssim f$.

\subsection{Notation for Couplings}

For a random variable $X$, denote by $\mathcal{L}(X)$ the distribution of $X$. We will use the following standard ``greedy" coupling between two Markov chains on a finite state space. See e.g. Proposition 4.7 of \cite{Levin2008MarkovCA} for a proof that the following construction yields well-defined processes.

\begin{definition} [Greedy Markovian Coupling] \label{EqGreedyMarkovianCoupling}
Fix a Markov transition kernel $K$ on finite state space $\Omega$. Note that for points $a, b \in \Omega$, it is possible to write
\begin{eqnarray*}
& K(a,\cdot) = \delta_{a,b} \, \mu_{a,b} + (1-\delta_{a,b}) \nu_{a,b; 1} \\
& K(b,\cdot) = \delta_{a,b} \,  \mu_{a,b} + (1-\delta_{a,b}) \nu_{a,b;2}
\end{eqnarray*}
where $\mu_{a,b}, \nu_{a,b;1}$ and $\nu_{a,b;2}$ are probability measures on $\Omega$ and
\[
\delta_{a,b} = 1 - d_{TV}(K(a,\cdot),K(b,\cdot)).
\]

Next, fix starting points $x, y \in \Omega$ and let $\{X_{t}\}_{t \geq 0}$ (respectively $\{Y_{t}\}_{t \geq 0}$) be a Markov chain evolving through $K$ and starting at $X_{0} = x$ (respectively $Y_{0} = y$).

The \textit{greedy Markovian coupling} of these chains is defined inductively by the following scheme for sampling $(X_{t+1},Y_{t+1})$ given $X_{t},Y_{t}$:

\begin{enumerate}
    \item Sample $B_{t} \in \{0,1\}$ according to the distribution $\mathbb{P}[B_{t}=1] = \delta_{X_{t},Y_{t}}$. Then:
    \begin{enumerate}
        \item If $B_{t}=1$, sample $Z \sim \mu_{X_{t},Y_{t}}$ and set $X_{t+1} = Y_{t+1} = Z$.
        \item If $B_{t}=0$, sample $X_{t+1} \sim \nu_{X_{t},Y_{t};1}$ and $Y_{t+1} \sim \nu_{X_{t},Y_{t};2}$ independently.
    \end{enumerate}
\end{enumerate}
\end{definition}

We note that the coupling in Definition \ref{EqGreedyMarkovianCoupling} has the following properties:

\begin{enumerate}
    \item The joint process $\{X_{t},Y_{t}\}_{t \geq 0}$ is a Markov chain.
    \item Set $\tau = \min\{t \, : \, X_{t} = Y_{t}\}$. Then $X_{s} = Y_{s}$ for all $s > \tau$ under this coupling.
\end{enumerate}

\subsection{Notation for Gibbs Samplers and Graphical Models }

Fix $q \in \mathbb{N}$ and define $[q] = \{1,2,\ldots,q\}$. We denote by the triple $G=(V,\Phi,E)$ a ``factor graph". Since the notation for ``factor graphs" does not seem completely standardized in statistics, we specify:

\begin{itemize}
    \item $V$ is any finite set. We call this set the \textit{vertices} of the graph.
    \item $E \subset \{(u,v) \, : \, u, v \in V, \, u \neq v\}$ is any set of pairs of distinct vertices. We call this set the \textit{edges} of the graph.
    \item $\Phi$ is any collection of functions from $\Omega \equiv [q]^{V}$ to $\mathbb{R}$. We call this set the \textit{factors} of the graph.
\end{itemize}

In this paper, we always assume the edge set is undirected: $(u,v) \in E$ if and only if $(v,u) \in E$. For a graph metric $d$, we abuse notation slightly by defining the distance between two vertex sets $A$ and $B$ as
\[
d(A,B)=\min\{d(x,y):x\in A,y\in B\}.
\]
For a set $A\subset V$ and integer $r>0$, we denote by $B_r(A)$ the set of vertices whose distance from $A$ is less than $r$. That is,
\[
B_r(A)=\{x\in V: d(x,A)< r\}.
\]

We define the Gibbs measure associated with a factor graph to be the probability measure on $\Omega$ given by:
\begin{equation}\label{DefEqFactorMsr}
\mu(\sigma)\propto\exp\left(\sum_{\phi\in \Phi}\phi(\sigma)\right),
\end{equation}
where $\sigma$ represents a configuration in $\Omega$.
Factors of a factor graph typically depend on only a few values. To make this precise, we say that a factor $\phi \in \Phi$ \textit{does not depend} on a vertex $v \in V$ if
\[
\phi(\sigma) = \phi(\eta)
\]
for all $\sigma, \eta \in \Omega$ satisfying
\[
\sigma(u) = \eta(u), \qquad u \neq v.
\]
Otherwise, we say $\phi$ \textit{depends} on $v$.
For $x\in V$, denote by
\[
A[x]=\{\phi \in \Phi: \mbox{factor $\phi$ depends on variable $x$}\}
\]
the set of factors that depend on variable $x$.
Similarly, let
\[
S[\phi] = \{x \in V  : \mbox{factor $\phi$ depends on variable $x$} \}.
\]

Given $A\subset V$, we denote by $\mu |_{A}$ the restriction of the measure $\mu$ to the set $A$, i.e.
\[
\mu |_{A}(\sigma)=\sum_{\eta(A)=\sigma(A)}\mu(\eta).
\]
For $A \subset B \subset V$ and $\sigma \in \Omega$, we denote by $\mu |_{A}^{\sigma |_{B^c}}$ the restriction to $A$ of the conditional distribution of $\mu$ given that it takes value $\sigma$ on $B^c$.
Throughout this paper, we impose the following constraints on the dependency among the vertices of the graphical model. Assume that $V$ can be partitioned into $\Gamma$ pairwise disjoint subsets $\{S_i\}$, and that the Gibbs measure $\mu$ possesses the following factorization structure:
\begin{equation}\label{tree-distributed}
\mu(\sigma)=\prod_{j=1}^{\Gamma}\mu|_{S_j}^{\sigma|_{\Pi_j}}(\sigma),
\end{equation}
where $\Pi_1=\emptyset$ and $\Pi_j\subset \bigcup_{i=1}^{j-1} S_i$, $2\leq j\leq \Gamma$, corresponds to the set of variables conditioned on which the configuration of $S_j$ is independent from everything else in $\bigcup_{i=1}^{j-1} S_i$.
Say that two measures have the same dependence structure if they can both be written in the form \eqref{tree-distributed} with the same lists $\{S_i\}$ and $\{\Pi_i\}$.

Denote by $\sigma_{x\mapsto i}$ be the configuration
\begin{align*}
\sigma_{x\mapsto i}(x)&=i, \\
\sigma_{x\mapsto i}(y) &=\sigma(y), \qquad y\neq x.
\end{align*}
We recall in Algorithm \ref{AlgUsualGibbs} the usual algorithm for taking a single step from a Gibbs sampler targeting Equation \eqref{DefEqFactorMsr}.

\begin{algorithm}[H]
\caption{Step of Standard Gibbs Sampler }\label{AlgUsualGibbs}
\begin{algorithmic}[1]
\Require Starting state $\sigma \in \Omega$.
\State Sample variable index $x \sim \mathrm{Unif}(V)$.
\State For $i \in [q]$, calculate $s_i=\sum_{\phi\in A[x]}\phi(\sigma_{x\mapsto i})$. Define the distribution $p$ over $[q]$ by:
\begin{equation}\label{EqGibbsDistS}
p(i)\propto \exp(s_i).
\end{equation}
\State  Sample $j \sim p$.
\State Return the element $\eta = \sigma_{x \mapsto j}$.
\end{algorithmic}
\end{algorithm}

This Gibbs sampler defines a Markov chain with transition matrix
\[
Q(\sigma,\sigma_{x\mapsto j})=\frac{1}{|V|}\cdot \frac{\exp(s_j)}{\sum_{i=1}^q \exp(s_i)}.
\]
In addition, we denote by $Q |_{B}$ the Gibbs sampler that only updates labels in $B$ and fixes the value of all labels in $B^c$. Note that $Q |_{B}$ with initial state $\sigma$ has stationary measure $\mu |_{B}^{ \sigma |_{B^c}}$.

We next define the class of algorithms that will be the focus of this paper:

\begin{definition}\label{perturbedGibbs}
Let $G=(V,\Phi,E)$ be a factor graph. Let $(\mathbb{A},d)$ be a Polish space and $\mathcal{F}_{\mathbb{A}}$ the associated $\sigma$-algebra. We say that a Markov chain $K$ on state space $\Omega \times \mathbb{A}$ is a \textit{perturbed Gibbs sampler} with factor graph G if it has the following properties:
\begin{enumerate}
    \item \textbf{Single-Site Updates:} For samples $(\eta,b) \sim K((\sigma,a),\cdot)$, we have that $|\{x \in V \, : \, \eta(x) \neq \sigma(x)\}| \leq 1$.
    \item \textbf{Gibbs Stationary Measure Respects Factorization Structure:} $K$  is ergodic, with stationary measure $\hat{\nu}$ whose marginal on $\Omega$ is of the form \eqref{DefEqFactorMsr} and can be expressed as \eqref{tree-distributed}.
\end{enumerate}
\end{definition}

\begin{remark} [Checking that a Gibbs Sampler Satisfies Definition \ref{perturbedGibbs}] \label{RemCheckCond}
The first part of Definition \ref{perturbedGibbs} is fairly innocuous, both in that it is often straightforward to check and in that our arguments are not terribly sensitive to small violations of this assumption. For example, it is straightforward to extend our results from single-site Gibbs updates to the case of Gibbs samplers that update $O(1)$ different entries that are all ``close" in the factor graph.

The second condition is more dangerous, in that it is both harder to check and our arguments do not work as written if it fails. Broadly speaking, we know of two ways to check that a transition kernel satisfies this condition:

\begin{enumerate}
    \item In the case that $K$ has the same state space as $Q$ (i.e. $\mathbb{A}$ has one element), the simplest sufficient conditions that we are aware of appear in Proposition 1 of \cite{GuyonHardouin02}. The result in \cite{GuyonHardouin02} requires that the Markov chain be ``synchronous" (a property that is straightforward to enforce for approximate Gibbs samplers based on subsampling ideas) and reversible (a familiar property). Although this condition is simple and seems strong enough for many purposes, it is far from the strongest possible condition; see e.g. \cite{KUNSCH1984159} for sufficient conditions that apply even to non-synchronous Markov chains.
    \item In the case that $K$ has the form of a pseudomarginal algorithm (see \cite{pseudo09} for an introduction to pseudomarginal algorithms and \cite{QuirSubSpeed19} for pseudomarginal algorithms in the context of approximate MCMC), an exact formula for the marginal of the stationary distribution on $\Omega$ is available. Thus, as long as the random log-likelihood estimate can be written in the form
    \begin{equation}\label{lestimator}
    \hat{L}(\sigma) = \sum_{\phi \in \Phi} \hat{L}_{\phi}(\sigma),
    \end{equation}
    where $\{\hat{L}_{\phi}(\sigma)\}_{\phi \in \Phi}$ are independent and $\hat{L}_{\phi}(\sigma)$ depends only on $\{\sigma(i)\}_{i \in S[\phi]}$, the stationary distribution will also be of the form \eqref{DefEqFactorMsr}. We use this condition in the worked example in Section \ref{minibatch_sec}.
\end{enumerate}

In our proof, the second condition of Definition \ref{perturbedGibbs} is used to invoke Inequality \eqref{IneqHellingerSubad} (originally proved as Theorem 1 in \cite{pmlr-v65-daskalakis17a}). The proof in \cite{pmlr-v65-daskalakis17a} relies on an exact factorization of the associated likelihoods, and it is beyond the scope of this paper to study target distributions that are merely ``close" to such a factorization.
\end{remark}

\section{Intuition and Simple Examples} \label{SecSimpleEx}

We give some simple (and far-from-optimal) calculations to show that simple perturbations are nearly sharp for generic examples, but far from sharp for structured examples. Our simple structured example also shows that our main result, Theorem \ref{mainThm}, is nearly sharp.

First, we recall Theorem 1 of \cite{negrea_rosenthal_2021}. To state it, we introduce the following temporary notation. Recall that a kernel $Q$ is said to be geometrically ergodic with factor $\rho$ if there exists a finite function $C$ such that for every initial distribution $\nu$ and for all $n \in \mathbb{N}$,
\begin{equation} \label{DefGeomErg-2}
d_{TV}(\mu, \nu Q^{n}) \leq C(\nu) \rho^{n},
\end{equation}
and is said to be $L^2(\lambda)$-geometrically ergodic if
\begin{equation} \label{DefGeomErg}
d_{L^2(\lambda)}(\mu, \nu Q^{n}) \leq C(\nu) \rho^{n}.
\end{equation}
As per Theorem 2 in \cite{Roberts-Rosenthal-97}, $Q$ is $L^2(\lambda)$-geometrically ergodic if and only if it is geometrically ergodic, and this equivalence holds for the same value of $\rho$ (but not necessarily the same value of $C(\nu)$).
In this notation, Theorem 1 of \cite{negrea_rosenthal_2021} states:

\begin{theorem}\label{IneqSimplePert}
Let $Q,K$ be two transition kernels with stationary measures $\mu,\nu$ and assume that $Q$ satisfies Inequality \eqref{DefGeomErg}. Then for $d_{L^2(\lambda)}(Q,K) \ll 1-\rho$,
\[
d_{L^2(\lambda)}(\mu,\nu)\leq \frac{d_{L^2(\lambda)}(Q,K)}{\sqrt{(1-\rho)^2-d^2_{L^2(\lambda)}(Q,K)}}\approx\frac{d_{L^2(\lambda)}(Q,K)} {1-\rho} \footnote{We point out that submultiplicativity of the total variation distance to stationarity, as in Lemma 4.11 of \cite{Levin2008MarkovCA}, implies that a finite state chain with finite mixing time is geometrically ergodic with factor $\rho$ for which $(1-\rho)^{-1} = O(\tau_{mix})$.} .
\]
\end{theorem}

In particular, we can see that $d_{L^2(\lambda)}(\mu,\nu)$ is small if $d_{L^2(\lambda)}(Q,K)$ is much smaller than the mixing rate $\min\{\tau_{mix}(Q)$,$\tau_{mix}(K)\}^{-1}$.
Qualitatively similar results, of the form
\begin{equation} \label{IneqGeneralPert}
d(\mu,\nu) \lesssim \, \tau_{mix}(Q,d) d(Q,K)
\end{equation}
are known to hold for many metrics $d$, as long as the right-hand side is sufficiently small. For example, see \cite{mitrophanov05} for this result with the choice $d = d_{TV}$. The following simple example shows that Theorem \ref{IneqSimplePert} is close to sharp:

\begin{example} \label{ExSharpnessSimple}
For two parameters $p\in (0,0.5)$ and $C \in (0,p^{-1})$, define the transition kernels $Q$ and $K$ on $\{1,2\}$ by the following non-holding probabilities:
\[
Q(1,2) = Q(2,1)= p,
\]
\[
K(1,2) =Cp,\quad  K(2,1)=p.
\]

We view $Q$ as the original chain, and its stationary distribution $\mu$ which is uniform serves as the reference measure in the $L^2$ distance. It is straightforward to verify the following scaling estimates for fixed $p$ and $C$ in the range $C \in (0.5,2)$:
\begin{enumerate}
\item The distance between kernels $d_{L^2(\mu)}(Q,K) \approx |C-1|p$.
\item   The mixing rates of both chains satisfy $\tau_{mix}(Q,d_{L^{2}(\mu)}) \approx  \tau_{mix}(K,d_{L^{2}(\mu)}) \approx p$.
\item The distance between stationary distributions $d_{L^2(\mu)}(\mu,\nu) \approx |C-1|$.
\end{enumerate}
By the above items, we can see that $d_{L^2(\mu)}(\mu,\nu)$ is of the same order as the error $\min\{\tau_{mix}(Q, d_{L^{2}(\mu)})$, $\tau_{mix}(K,d_{L^{2}(\mu)})\} \, d_{L^2(\mu)}(Q,K)$ for fixed $p$ as $C\rightarrow 1$.
\end{example}

Example \ref{ExSharpnessSimple} shows that the error upper bound in Theorem \ref{IneqSimplePert} scales at best like the product $\min\{\tau_{mix}(Q),\tau_{mix}(K)\} \, d_{L^2(\mu)}(Q,K)$ when this product is small. When the product is large, the error can actually ``blow up." The following example illustrates this ``blowing up" in the simpler total variation norm:

\begin{example} \label{ExSharpnessSimple-tv}
For $n \in \mathbb{N}$ and $p \in (0,0.5)$, define the transition kernel $Q = Q_{n,p}$ on $[n]$ by:
\begin{align*}
Q(i,i+1) &= \frac{1}{2}p, \qquad  \qquad  \, \, \, i < n,\\
Q(i,i-1) &= \frac{1}{2}(1-p), \qquad  i >1,\\
Q(i,i) &= \frac{1}{2}, \qquad \qquad \, \, \, \, \, \, 1 < i < n
\end{align*}
and $Q(1,1) = \frac{1}{2}(2-p)$, $Q(n,n) = \frac{1}{2}(1+p)$. Next, for $0 < \epsilon <1$, define $K = K_{n,p,\epsilon}$ by
\begin{equation}\label{k-perturbation}
K(i,j) = (1-\epsilon) Q(i,j) + \epsilon \mathbf{1}_{\{j=n\}}.
\end{equation}
Informally: $K$ takes a step from $Q$ with probability $1-\epsilon$ and teleports to $n$ with probability $\epsilon$.

We consider the regime where $p \in (0,0.5)$ is fixed and $\epsilon = \epsilon_{n} = \frac{C_{n}}{n}$ for some sequence $C_{n} \rightarrow \infty$. The proofs of the following four facts can be found in Appendix \ref{AppSharpExPfs}:
\begin{enumerate}
\item The distances between kernels $d_{TV}(Q,K) \approx \frac{C_{n}}{n}$.
\item   The mixing times satisfy $\tau_{mix}(Q) \approx n$, $\tau_{mix}(K) \gtrsim \frac{n}{C_n}$.
\item The stationary distribution $\mu$ of $Q$ assigns probability $\mu([0,\frac{n}{3}]) \rightarrow 1$ as $n \rightarrow \infty$.
\item The stationary distribution $\nu$ of $K$ assigns probability $\nu([\frac{2n}{3},n]) \rightarrow 1$ as $n \rightarrow \infty$.
\end{enumerate}

By items 3 and 4, we can see that $d_{TV}(\mu,\nu) \rightarrow 1$, so the stationary measures are very far apart. On the other hand, by items 1 and 2, the error $\min\{\tau_{mix}(Q)$, $\tau_{mix}(K)\} \, d_{TV}(Q,K) \lesssim C_{n}$ grows arbitrarily slowly.
\end{example}

 However, as we will see, Inequality \eqref{IneqGeneralPert} is far from sharp for many special cases:

 \begin{example} \label{ExSharpnessProduct}
Consider $0.01 < p < \tilde{p} < 0.99$. Let  $\mu$ be the distribution of random vector $X=(X_1,X_2,\ldots,X_{n^2})$ on $\{0,1\}^{n^2}$, where all the random variables $X_i$, $i=1,2,\ldots,n^2$ are independently and identically distributed and take value over Bernoulli distribution with parameter $p$. Denote by $\nu$ the analogous distribution with parameter $\tilde{p}$. Denote by $Q$, $K$ the usual Gibbs samplers with targets $\mu,\nu$. It is well-known that their mixing times are $\tau_{mix}(Q), \tau_{mix}(K) \approx n^{2} \log(n)$, uniformly in the choice of $0.01 < p < \tilde{p} < 0.99$. It is straightforward to check that $d_{TV}(Q(x,\cdot),K(x,\cdot)) \approx |\tilde{p}-p|$. For $0<p<\tilde{p}<1$, formula (2.15) in \cite{adell05} 
 says
\[
d_{TV}(\mu,\nu)\leq \frac{\sqrt{e}}{2}\cdot \frac{C(\tilde{p}-p)}{(1-C(\tilde{p}-p))^2},
\]
with
\[
C(x):=x\sqrt{\frac{n^2+2}{2p(1-p)}}.
\]
This shows that we can allow perturbations of size up to $d_{TV}(Q,K) \lesssim n^{-1}$, even though the mixing times are $\tau_{mix}(Q), \tau_{mix}(K) \approx n^{2} \log(n)$. This motivates us to generalize this behavior and find a more relaxed condition on the perturbation to ensure robustness.
 \end{example}

We think of the main result of this paper as an extension of the phenomenon in Example \ref{ExSharpnessProduct}. To give a hint as to how one might prove this, we give an informal calculation for a simple graphical model. Take the Ising model at a high enough temperature on a two-dimensional lattice for example. The decay-of-correlation property yields that for any sequence $\omega_{n} \gtrsim \log(n)$, mixing occurs at a point $x$ before any influence from outside of the surrounding box of side-length $O(\omega_{n})$ can propagate to that point. This means that, when looking at perturbation bounds, we can effectively focus on a box growing at rate $\log(n)^{2}$ rather than the usual $n^{2}$. Thus, naive estimates with perturbation size up to $O(\log(n)^{-2})$ are enough to tell us that the marginal distributions at the center of the box are close. If we can leverage the independence property, one expects to be able to get bounds on the whole chain that are qualitatively similar to those in Example \ref{ExSharpnessProduct}.

We will use the following subadditivity property of the Hellinger distance, proved in \cite{pmlr-v65-daskalakis17a}, to leverage this independence property:
\begin{lemma}\label{subaddHellinger}
If $\mu$ and $\nu$ are two Markov random fields on $V$ with common factorization structure as (\ref{tree-distributed}),
then we have
\begin{equation}\label{IneqHellingerSubad}
d_H^2(\mu,\nu) \leq \sum_{j=1}^{\Gamma} d_H^2(\mu |_{S_j\cup \Pi_j},\nu |_{S_j\cup \Pi_j}).
\end{equation}
\end{lemma}

\section{General Results}\label{mainresult}

Throughout this section, we assume that $\mu$ is the Gibbs measure of a factor graph $G=(V,\Phi,E)$ on state space $\Omega = [q]^{V}$ with factorization structure (\ref{tree-distributed}), that $Q$ is the associated Gibbs sampler, and that $K$ is a perturbed Gibbs sampler with the same factor graph in the sense of Definition \ref{perturbedGibbs}. We set notation for sampling from the first component of $K$: if $(\eta,b) \sim K((\sigma,a),\cdot)$, we say that $\eta \sim K((\sigma,a),\cdot)|_{\Omega}$. We also denote by $\widehat{\nu}$ the stationary measure of $K$ and $\nu$ its marginal distribution on $\Omega$.


Fix a metric $d$ on probability measures. Our results rely on the following four quantitative assumptions:

\begin{assumption}[Decay of Correlations]\label{asum1}
There exist positive constants $ m, C_{1} < \infty$ such that for all configurations $\sigma,\eta \in \Omega$, all $j\in [\Gamma]$ and any $r>0$:
\[
d_H(\mu |_{S_j\cup \Pi_j}^{\sigma |_{B_r^c(S_j\cup \Pi_j)}}, \mu |_{S_j\cup \Pi_j}^{\eta |_{B_r^c(S_j\cup \Pi_j)}}) \leq C_{1}e^{-mr}.
\]
\end{assumption}

\begin{assumption}[Relationship to Hellinger]\label{asum4}
There exists a positive constant $C_2<\infty $ such that for any probability measures $\mu,\nu$ on $\Omega$
\[
d_H(\mu,\nu) \leq C_2 d(\mu,\nu).
\]
\end{assumption}

\begin{assumption}[Propagation of Perturbations]\label{asum2}
There exists a positive constant $ C_{3} < \infty$ and a strictly increasing function $f:\mathbb{R}_+\rightarrow \mathbb{R}_+$ so that for all configurations $\sigma \in \Omega$ and all $j\in [\Gamma]$ and any $r>0$:
\begin{equation}\label{ineq2}
d(\mu|_{S_j\cup \Pi_j}^{ \sigma |_{B_r^c(S_j\cup \Pi_j)}}, \nu |_{S_j\cup \Pi_j}^{ \sigma |_{B_r^c(S_j\cup \Pi_j)}}) \leq C_{3}\, f(r) \sup_{(\sigma,a) \in \Omega \times \mathbb{A}} d(Q(\sigma,\cdot), K((\sigma,a),\cdot)|_{\Omega}).
\end{equation}
\end{assumption}

\begin{assumption}[Small Perturbations]\label{asum3}
There exists a positive constant $C_{4} < \infty$ so that
\begin{equation}\label{ineq3}
\sup_{(\sigma,a) \in \Omega \times \mathbb{A}} d(Q(\sigma,\cdot), K((\sigma,a),\cdot)|_{\Omega}) \leq \frac{C_{4}}{\sqrt{\Gamma} f(\log\Gamma)}.
\end{equation}
\end{assumption}

We comment on these assumptions, in order:

\begin{enumerate}
    \item Assumption \ref{asum1}, also known as strong spatial mixing, plays a key role in the study of exponential ergodicity of Glauber dynamics for discrete lattice spin systems (see Section 2.3 in \cite{m99} for details). It implies that a local modification of the boundary condition has an influence on the corresponding Gibbs measure that decays exponentially fast with the distance from the site of the perturbation. There is a large literature devoted to studying the regimes where strong spatial mixing holds for particular models. In particular, the Ising and Potts models \cite{mos94, mo94, salas97, bubley99, goldberg05, goldberg06,ding22}, the hard-core model \cite{weitz06}, and $q$-colorings \cite{gamarnik12} have received special attention.
    \item Assumption \ref{asum4} places restrictions on the relationship between distance measure $d$ and the Hellinger distance $d_{H}$. Many popular norms satisfy this condition. For example, by Lemma 20.10 in \cite{Levin2008MarkovCA} the $L^{2}$ norm satisfies the condition with $C_2=1$.
    \item Assumption \ref{asum2} is a perturbation bound for discrete-time Markov chains, specialized to local subsets. In our settings where the state space is finite, the discrete-time Markov chain corresponding to the Gibbs sampler is strongly ergodic.  \cite{mitrophanov05} shows bounds of this form, with $C_{3}$ being closely related to the mixing time of $Q|_{B_r(S_j\cup \Pi_j)}$. See also \cite{johndrow2017error,negrea_rosenthal_2021,rudolf18} for closely-related estimates.
    \item Assumption \ref{asum3} is the main assumption of our result, and must be checked for individual approximate chains. We think of this assumption as a nearly-sharp criterion for obtaining small error in the stationary measure, and we expect that users should \textit{design} their MCMC algorithms to satisfy this criterion.
\end{enumerate}


We are now in a position to state our main result:

\begin{theorem}\label{mainThm}
Let the sequences of measures $\mu = \mu^{(\Gamma)}$ and $\nu = \nu^{(\Gamma)}$ have common factorization structure (\ref{tree-distributed}), and their associated Gibbs samplers $Q = Q^{(\Gamma)},K = K^{(\Gamma)}$ satisfy Assumptions \ref{asum1}-\ref{asum3} with constants $C_{1},C_{2},C_{3}$ and $C_4$ that do not depend on $\Gamma$. Then
\[
\lim_{\Gamma \rightarrow \infty} d_H(\mu^{(\Gamma)},\nu^{(\Gamma)}) = 0.
\]
\end{theorem}

\begin{proof}
Under these assumptions, we have the following basic calculation for any fixed $\Gamma$ (which we omit in the notations for simplicity), any $j\in [\Gamma]$ and any $r>0$:
\begin{eqnarray*}
&&d_H(\mu |_{S_j\cup \Pi_j}^{\sigma |_{B_r^c(S_j\cup \Pi_j)} }, \nu |_{S_j\cup \Pi_j}^{\eta |_{B_r^c(S_j\cup \Pi_j)}}) \\
&\leq& d_H(\mu |_{S_j\cup \Pi_j}^{\sigma |_{B_r^c(S_j\cup \Pi_j) }}, \mu |_{S_j\cup \Pi_j}^{\eta |_{B_r^c(S_j\cup \Pi_j)}}) + d_H(\mu |_{S_j\cup \Pi_j}^{\eta |_{B_r^c(S_j\cup \Pi_j) }}, \nu |_{S_j\cup \Pi_j}^{\eta |_{B_r^c(S_j\cup \Pi_j)}}) \\
&\leq& C_{1}e^{-mr} + C_2 d(\mu |_{S_j\cup \Pi_j}^{\eta |_{B_r^c(S_j\cup \Pi_j)} }, \nu |_{S_j\cup \Pi_j}^{\eta |_{B_r^c(S_j\cup \Pi_j) }}) \\
&\leq& C_{1}e^{-mr} + C_{2} C_3 f(r)\sup_{(\sigma,a) \in \Omega \times \mathbb{A}} d(Q(\sigma,\cdot), K((\sigma,a),\cdot)|_{\Omega}).
\end{eqnarray*}
We now consider what happens as the number $\Gamma$ of partitions gets large. Let $M> 0$ be a constant to be chosen later and $r = M \, \log\Gamma$. Using Assumption \ref{asum3}, we have
\[
d_H(\mu |_{S_j\cup \Pi_j}^{\sigma |_{B_r^c(S_j\cup \Pi_j)} }, \nu |_{S_j\cup \Pi_j}^{\eta |_{B_r^c(S_j\cup \Pi_j)}}) \leq C_1e^{-m M \log\Gamma} + \frac{ C_{2} C_{3} C_4 f(M \log\Gamma)}{\sqrt{\Gamma}f(\log\Gamma)}.
\]

In particular, for any fixed $\delta > 0$, we can choose $M$ small enough so that
\[
d_H(\mu |_{S_j\cup \Pi_j}^{\sigma |_{B_r^c(S_j\cup \Pi_j)} }, \nu |_{S_j\cup \Pi_j}^{\eta |_{B_r^c(S_j\cup \Pi_j)}}) \leq \frac{\delta}{\sqrt{\Gamma}}
\]
for all sufficiently large $\Gamma$.
Plugging this into \eqref{IneqHellingerSubad} gives
\begin{eqnarray*}
d_H^2(\mu,\nu) &\leq& \sum_{j=1}^{\Gamma} d_H^2(\mu |_{S_j\cup \Pi_j},\nu |_{S_j\cup \Pi_j})
\leq \delta^{2},
\end{eqnarray*}
which allows us to complete the proof by the arbitrariness of $\delta$.
\end{proof}

\begin{remark}
In the special case where $\Gamma=|V|$, $d=d_{L^2}$ and $f$ is polynomial, this shows that we have ``small" perturbation error in the Hellinger distance even with ``large" kernel perturbations up to order $(\sqrt{|V|} \text{polylog}(|V|))^{-1}$, even though the sample space is $|V|$-dimensional and the mixing time is about $|V|$. In other words, compared with generic bounds such as \cite{mitrophanov05, rudolf18} as summarized in our simple heuristic bound \eqref{IneqGeneralPert}, our Assumption \ref{asum3} allows for much larger perturbations.
\end{remark}

\section{Application to Pseudo-marginal Algorithms}\label{minibatch_sec}


 The main goal of this section is to illustrate one way to design an algorithm for which our perturbation result applies. As the example shows, this is not terribly difficult, but does require some effort. Our main tool for designing an algorithm of the correct form is the pseudo-marginal algorithm of \cite{pseudo09}.

Note that this worked example is meant only to illustrate how to design the overall MCMC algorithm, \textit{given} plug-in estimators for the likelihood. It is \textit{not} meant to illustrate how to design good plug-in estimators. The latter task is well-studied in the context of perturbed MCMC chains, but the best estimators are often problem-specific and giving advice is beyond the scope of the current paper. See \textit{e.g.} \cite{Quiroz18} for practical mainstream advice, largely based around choosing good control variates.

The worked example in this section is based on the problem of inferring values sampled from a latent hidden Markov random field given observations at each of its vertices. We begin with the underlying hidden Markov random field. Throughout this section, denote by $G=(V,\Phi_1,E)$ a factor graph associated with state space $\Omega = [q]^{V}$ and Gibbs measure $f$ with factorization structure (\ref{tree-distributed}). Next, we define the distributions for the observations at each vertex. Denote by $\{e^{\ell(i,\cdot)}\}_{ i \in [q]}$ a family of distributions on $[q]$, indexed by the same set $[q]$. For each vertex $v \in V$, fix $n_{v} \in \mathbb{N}$; this represents the ``number of observations" associated with vertex $v$.

This data lets us define the joint distribution of a sample from the hidden Markov random field and the associated observations. First, the latent sample from the hidden Markov random field:
\[
\sigma \sim f.
\]
Conditional on $\sigma$, sample independently the observations
\[
Z_{ij} \stackrel{ind}{\sim} e^{\ell(\sigma(i),\cdot)},\quad j\in [n_i].
\]
This completes the description of the data-generating process. The usual statistical problem is to infer $\sigma$ given the observed $Z$.

Given the independent observations  $Z=\{z_{ij}\}_{ i\in V, j\in [n_i]}$ with $Z_{i} = \{z_{ij}\}_{j \in [n_{i}]}$, we denote the log-likelihood of vertex $i$ by $\phi_i(\sigma ,Z_{i})=\sum_{j=1}^{n_i}\ell(\sigma(i),z_{ij})$ and define $\Phi_{2} = \{ \phi_{i}\}$. Note that our model on the random variables $(\sigma,Z)$ is then equivalent to a Markov random field on larger vertex set $V \cup \{ (i,j)\}_{i \in V, j \in [n_i]}$ and with augmented collection of factors $\Phi = \Phi_{1} \cup \Phi_{2}$. The associated posterior distribution on the hidden $\sigma$ given the observed $Z$ is:

\[
\mu(\sigma|Z)\propto f(\sigma) \cdot \exp\left(\sum_{\phi_i\in \Phi_2}\phi_i(\sigma,Z_{i})\right).
\]

When many values of $n_{i}$ are large, an exact evaluation of the likelihood is typically expensive. A natural approach is to rely on a moderately sized random subsample of the data in each step of the algorithm.
Let $a:=(a_i)_{i \in V} \in \mathbb{A}$ be a vector of auxiliary variables with each element $a_i$ a subset of $[n_{i}]$; we think of this as corresponding to the subset of observations to include when estimating $\phi_i(\sigma,Z_i)$. Fix a probability distribution $g$ on $\mathbb{A}$ of the form
\[
g(a)=\prod_{i\in V}g_{i}(a_{i}).
\]
We next define an estimator $\hat{L}_{\phi}(\sigma,a;Z)$ for each element $\phi \in \Phi$:
\begin{equation}\label{hatL}
\hat{L}_{\phi} (\sigma,a;Z) = \left\{
\begin{array}{ll}
\phi(\sigma), & \phi \in \Phi_1, \\
\frac{n_i}{|a_i|}\sum_{j\in a_i}\ell(\sigma(i),z_{ij}), & \phi=\phi_i \in \Phi_2.
\end{array}
\right.
\end{equation}
The associated approximate posterior distribution on the augmented space $\Omega \times \mathbb{A}$ is then
\begin{equation}\label{post_pseudo}
\hat{\nu}(\sigma,a|Z)\propto g(a)\exp\left(\sum_{\phi\in \Phi}\hat{L}_{\phi}(\sigma,a; Z)\right).
\end{equation}

Algorithm \ref{AlgUsualAltGibbs} below gives a step of an  alternating Gibbs sampler that targets (\ref{post_pseudo}) in algorithmic form.  The process can be succinctly described as follows: it first performs Gibbs sampler $K^{(a)}$ on $\sigma$ targeting $\hat{\nu}(\sigma|a,Z)$, and then conducts Gibbs sampler $K^{(\eta)}$ on $a$ by updating one observation in the subset of some vertex given updated configuration $\eta$ from the previous step. This leads to the alternating Gibbs sampler, denoted by $K$, which can be represented as $K=K^{(a)}\cdot K^{(\eta)}$.

Notably, Algorithm \ref{AlgUsualAltGibbs} gives  a perturbed Gibbs sampler in the sense of Definition \ref{perturbedGibbs} (see part 2 of Remark \ref{RemCheckCond} for explanation). Since $g$ is a product measure on $\prod_{i\in V}\mathbb{A}_{i}$,
\begin{eqnarray*}
\nu(\sigma|Z) &\propto& \int_{\mathbb{A}}g(a)\prod_{\phi \in \Phi} \exp(\hat{L}_{\phi}(\sigma,a; Z))da \\
&=& f(\sigma) \cdot \int_{\mathbb{A}}g(a)\prod_{\phi_i \in \Phi_2}\exp(\hat{L}_{\phi_i}(\sigma,a;Z))da \\
&=& f(\sigma) \cdot \prod_{\phi_i \in \Phi_2}\int_{\mathbb{A}_i}(\exp(\hat{L}_{\phi_i}(\sigma,a;Z))+\log(g_i(a_i)))da_i,
\end{eqnarray*}
which is of the form (\ref{DefEqFactorMsr}). Moreover, it is obvious that $f$, $\mu(\sigma|Z)$ and $\nu(\sigma|Z)$ share the same dependence structure since they only differ in the product terms related to $\phi_i\in \Phi_2$,  which do not depend on other vertices except some single vertex. This ensures that if the prior $f$ exhibits a factorization structure as \eqref{tree-distributed}, then the same holds true for $\mu(\sigma|Z)$ and $\nu(\sigma|Z)$.

\begin{algorithm}[H]
\caption{Step of Alternating Gibbs Sampler.}\label{AlgUsualAltGibbs}
\begin{algorithmic}[1]
\Require Initial state $(\sigma,a)\in \Omega \times \mathbb{A}$,  distribution $g$ on $\mathbb{A}$ and observations $Z$.
\State Sample variable index $v \sim \mathrm{Unif}(V)$.
\State For $i \in [q]$, calculate $s_i^a=\sum_{{\phi}\in A[v]} \hat{L}_{\phi}(\sigma_{v \mapsto i},a;Z)$, and construct distribution $p_1$ over $[q]$ according to:
\begin{equation}\label{EqGibbsDistp1}
p_1(i)\propto \exp(s_i^a).
\end{equation}
\State Sample $k \sim p_1$, and define $\eta$ as
\[
\eta(w)=
\left\{
\begin{array}{ll}
k, & w=v, \\
\sigma(w), & w \neq v.
\end{array}
\right.
\]
\State Sample variable index $v^{\prime} \sim \mathrm{Unif}(V)$, and choose two observations $z_1 \sim \mathrm{Unif}(a_{v^{\prime}})$ and $z_2 \sim \mathrm{Unif}(n_{v^{\prime}} \setminus a_{v^{\prime}})$.
\State For $i\in \{1,2\}$, calculate $l_i^{\eta}=\hat{L}_{\phi_{v^{\prime}}}(\eta,a^{(i)}_{v^{\prime}};Z)$ where $a_{v^{\prime}}^{(1)}=a_{v^{\prime}}$ and $a_{v^{\prime}}^{(2)}=a_{v^{\prime}}\cup \{z_2\} \setminus \{z_1\}$, and define distribution $p_2$ over $\{1,2\}$ by:
\[p_2(i)\propto g_{v^{\prime}}(a_{v^{\prime}}^{(i)}) \exp( l_i^{\eta} ).\]
\State Draw $k \sim p_2$, and let $b=(b_{w})$ where
\[
b_{w}=\left\{
\begin{array}{ll}
a_{v^{\prime}}^{(k)}, & w=v^{\prime}, \\
a_{w}, &  w \neq v^{\prime}.
\end{array}
\right.
\]
\State  Return $(\eta,b)$.
\end{algorithmic}
\end{algorithm}

\begin{remark} [Why such a small move on such a large augmented state space?]

In Algorithm \ref{AlgUsualAltGibbs}, steps 4 to 6 involve swapping out a \textit{single} element $z_{1}$ from a \textit{single} set $a_{v'}$. It is natural to ask: why not make a larger Gibbs move, perhaps regenerating \textit{all} elements from $a_{v'}$ or perhaps even \textit{all} elements from \textit{all} such subsets?

In principle, it is straightforward to write down an algorithm that would do this. However, even regenerating a single set $a_{v'}$ would require an enormous amount of work - the analogue of the distribution $p_{2}$ in step 5 would have support on a set of size ${n_{v'} \choose |a_{v'}|}$ rather than 2. Unless $|a_{v'}| = 1$, this grows much more quickly than the cost of simply keeping the full set of observations at all times.

Conversely, updating a single entry as we do in steps 4-6 can be very quick. In the extreme case that $p_{2}$ is always uniform, standard results on the mixing time of a Gibbs sampler on the hypercube indicate that $a_{v'}$ can be resampled in $O(n_{v'} \log(n_{v'}))$ steps, regardless of $a_{v'}$.

While this paper is not focused on providing ``near-optimal" perturbed algorithms, this choice allows us to provide an algorithm that has reasonable performance in the regime of interest.
\end{remark}

\subsection{Worked Example: A Tree-Based Ising Model}

In the context of the above Gibbs samplers, we show that it is possible to apply Theorem \ref{mainThm} to a particular family of graphs and distributions $f, \ell$. This amounts to verifying Assumption \ref{asum1}, since (as noted immediately after they are stated) Assumption \ref{asum4} holds for $d=d_{L^2}$ and it is straightforward to force Assumptions \ref{asum2} and \ref{asum3} by choosing a sufficiently good approximate likelihood $\hat{L}$ and sufficiently-large subsamples.

Unless otherwise noted, we keep the notation used throughout Section \ref{minibatch_sec}.

\begin{example} [Tree-Based Ising Model] \label{ExTreeIsing}
Fix $n \in \mathbb{N}$ and let $T_{n} = T = (V,E)$ be the usual depth-$n$ binary tree with $\sum_{j=0}^{n} 2^{j}$ vertices, labelled as follows:
\begin{center}
\scalebox{0.8}{
\begin{forest}
  for tree={
    circle,
    draw,
    minimum size=0.5em,
    math content,
  }
  [x_{00}
    [x_{11}
      [x_{21}]
      [x_{22}]
    ]
    [x_{12}
      [x_{23}]
      [x_{24}]
    ]
  ]
 \node[below=2ex, inner sep=0pt] at (current bounding box.south) {$\vdots$};
\end{forest}
}
\end{center}
We take the prior distribution $f$ as the Gibbs measure of the Ising model associated with fixed inverse temperature $\beta \in (0,\infty)$: each configuration $\sigma\in \Omega=\{\pm 1\}^{|V|}$ is assigned a prior probability
\[
f(\sigma)\propto \exp\left(\beta\sum_{(v,w) \in E}\sigma(v)\sigma(w)\right).
\]
Next, we fix a ``noise" parameter $0 < \delta < \frac{1}{2}$ and define the log-likelihood function associated with each node as:
\[
\ell(i,j) =  \textbf{1}_{\{i = -j\}} \log \delta + \textbf{1}_{\{i = j\}}\log (1-\delta) .
\]
That is, we have the set of observations $Z=\{Z_{ij}:i\in V, j\in [n_i]\}$, where
\[
Z_{ij}=\left\{
\begin{array}{ll}
\sigma(i), & \mbox{with probability $1-\delta$,}\\
-\sigma(i), & \mbox{with probability $\delta$.}
\end{array}
\right.
\]
\end{example}

The remainder of this section is an analysis of Example \ref{ExTreeIsing}. We denote the likelihood of observations on node $i$ by
\[
h_i(Z|\sigma(i))=\prod_{j\in [n_i]} e^{\ell(\sigma(i),z_{ij})} =\delta^{\frac{1}{2}\sum_{j\in [n_i]}|\sigma(i)-Z_{ij}|}\cdot (1-\delta)^{n_i-\frac{1}{2}\sum_{j\in [n_i]}|\sigma(i)-Z_{ij}|}.
\]
The posterior distribution on $\Omega$ is
\begin{eqnarray*}
&&\mu(\sigma|Z) \propto f(\sigma) \cdot \prod_{i\in V}h_i(Z|\sigma(i)) \\
&=& f(\sigma)\cdot \delta^{\frac{1}{2}\sum_{i\in V}\sum_{j\in [n_i]}|\sigma(i)-Z_{ij}|}\cdot (1-\delta)^{\sum_{i\in V}(n_i-\frac{1}{2}\sum_{j\in [n_i]}|\sigma(i)-Z_{ij}|)} \\
&=& \exp\left(\beta\sum_{(v,w) \in E}\sigma(v)\sigma(w) +\sum_{i\in V}\left(\frac{1}{2}\sum_{j\in [n_i]}|\sigma(i)-Z_{ij}|\log\frac{\delta}{1-\delta}+n_i\log(1-\delta)\right)\right),
\end{eqnarray*}
which is essentially the Gibbs measure of the Ising model with some external field.

Due to the tree structure of the graph, it is easy to check that $\mu(\sigma|Z)$ factorizes as
\begin{equation}\label{poterior-structure}
\mu(\sigma|Z)=\mu|_{x_{00}}(\sigma)\cdot \prod_{i=1}^{n} \prod_{j=1}^{2^{i-1}} \prod_{k=0}^1 \mu|_{x_{i,2j-k}}^{\sigma|_{x_{i-1,j}}}(\sigma),
\end{equation}
where
\[
\mu|_{x_{i,2j-k}}^{\sigma|_{x_{i-1,j}}}(\sigma_{x_{i,2j-k}\mapsto \pm 1})=\frac{e^{\pm \beta\sigma(x_{i-1,j})}h_{x_{i,2j-k}}(Z|\pm 1)}{e^{\beta\sigma(x_{i-1,j})}h_{x_{i,2j-k}}(Z|1)+e^{-\beta\sigma(x_{i-1,j})}h_{x_{i,2j-k}}(Z|-1)}.
\]
Hence, in this example we take $\{S_i\}$ as an enumeration of the vertices of $T$ and take $S_i\cup \Pi_i$ as the vertex $S_{i}$ and its parent.

Say that a sequence of vertices $\gamma_{1},\ldots,\gamma_{k} \in V$ is a \textit{path} if $(\gamma_{i},\gamma_{i+1}) \in E$ for each $i \in [k-1]$. In view of the factorization structure of \eqref{poterior-structure}, we make the following observation:

\begin{proposition} \label{PropMarkovProperty}
Fix a path $\gamma_{1},\gamma_{2},\ldots,\gamma_{k}$ in the tree $T$. Sample $\sigma,Z$ as in Example \eqref{ExTreeIsing}. Then, conditional on $Z$, the sequence $\sigma(\gamma_{1}),\sigma(\gamma_{2}),\ldots,\sigma(\gamma_{k})$ is a (time-inhomogeneous) Markov chain.
\end{proposition}

This lets us check:

\begin{lemma}
Fix a path $\gamma_{1},\gamma_{2},\ldots,\gamma_{k}$ in the tree $T$. Sample $Z$ as in Example \eqref{ExTreeIsing}. Then sample $\sigma^{(+)}$ (respectively $\sigma^{(-)}$) from the distribution in  Example \eqref{ExTreeIsing}, conditional on both (i) the value $Z$ and (ii) the value $\sigma(\gamma_{1}) = +1$ (respectively $\sigma(\gamma_{1}) = -1$).

Then:
\[
d_{TV}(\mathcal{L}(\sigma^{(+)}(\gamma_{j}), \ldots, \sigma^{(+)}(\gamma_{k})), \mathcal{L}(\sigma^{(-)}(\gamma_{j}), \ldots, \sigma^{(-)}(\gamma_{k}))) \leq (1-\epsilon_{\beta})^{j-1},
\]
where $\epsilon_{\beta} = e^{-6\beta} $.
\end{lemma}

\begin{proof}
To simplify notation, let $X_{t} = \sigma^{(+)}(\gamma_{t+1})$ and $Y_{t} = \sigma^{(-)}(\gamma_{t+1})$. We couple these according to the greedy coupling from Definition \ref{EqGreedyMarkovianCoupling} and let $\tau = \min\{t \, : \, X_{t} = Y_{t}\}$ be the usual coupling time.

For $j\geq 1$, we can calculate:
\begin{eqnarray*}
\mathbb{P}(\tau=j|\tau>j-1)&=&\mathbb{P}(X_{j}= Y_{j} | X_{j-1}\neq Y_{j-1})\\
&\geq& \min_{\sigma \in \Omega} (1 - d_{TV}(\mu(\sigma_{j \mapsto 1}|Z), \mu(\sigma_{j \mapsto -1}|Z))),
\end{eqnarray*}
where the last inequality comes from looking at the worst-case distributions of times after $j$ and the optimal coupling between them. Noting that each vertex in the tree has degree at most 3, we have
\[
e^{-6\beta} \leq \frac{\mu(\sigma_{j \mapsto +1} | Z)}{\mu(\sigma_{j \mapsto -1} | Z)} \leq e^{6 \beta}.
\]
Applying Lemma \ref{calculation-tv} from the appendix,
\begin{equation}
\mathbb{P}(\tau=j|\tau>j-1) \geq 1 -  \frac{1- e^{-6\beta}}{ 1+ e^{-6\beta}} = \frac{2 e^{-6 \beta}}{1 + e^{-6 \beta}} \geq e^{-6 \beta} = \epsilon_{\beta}.
\end{equation}
Continuing by induction on $j$, this yields
\[
\mathbb{P}(\tau>t)\leq (1-\epsilon_{\beta})^t,\ \ \ \mbox{for $t>0$}.
\]
In view of the property that $X_s=Y_s$ for all $s>j$ if $X_j=Y_j$, the claim then follows from the standard coupling inequality (see Theorem 5.4 of \cite{Levin2008MarkovCA}).
\end{proof}

Note that the boundary of $B_r(S_j\cup \Pi_j)$ in the tree $T$ has at most $2^{r+1}$ vertices. Thus, decay-of-correlation Assumption \ref{asum1} holds as long as $1-\epsilon_{\beta} < \frac{1}{2}$.

\section{Acknowledgments}
The first and second authors were supported by the National Natural Science Foundation of China (Grant No. 11971486), Natural Science Foundation of Hunan (Grant No. 2020JJ4674) and the Innovation Program of Hunan Province (Grant No. CX20220259).
The third author was supported by the NSERC Discovery Grant Program.

\normalem
\bibliographystyle{plain}
\bibliography{ref}

\appendix

\section{Appendix}

\subsection{Proof Sketches for Example \ref{ExSharpnessSimple-tv}} \label{AppSharpExPfs}

We give short proof sketches for the four claims in Example \ref{ExSharpnessSimple-tv}, in order. \\

\textbf{Claim 1:} The distances between kernels is $d_{TV}(Q,K) \approx \frac{C_{n}}{n}$.

\begin{proof}
Since $Q(i,\cdot),K(i,\cdot)$ are supported on at most 3 points, this is a straightforward calculation. For $i \neq n$,
\[
d_{TV}(Q(i,\cdot),K(i,\cdot)) = 2 \epsilon \approx \frac{C_{n}}{n}.
\]
For $i=n$ a similar calculation holds.
\end{proof}

\textbf{Claim 2:} The mixing times are $\tau_{mix}(Q) \approx n$, $\tau_{mix}(K) \gtrsim \frac{n}{C_n}$.

\begin{proof}
Note that there is a monotone coupling of two walks $X_{t}, Y_{t} \sim Q$ with $X_{t}$ started at $X_{0}=n$ and $Y_{t}$ started at any point $Y_{0}$. That is, it is possible to couple these chains so that $X_{t} \geq Y_{t}$ for all $t$.

Given such a monotone coupling, a standard coupling argument says that $\tau_{mix} \leq 4 \mathbb{E}[\min\{t \, : \, X_{t} = n\}]$. Since $Q$ is the standard simple random walk with drift, it is well-known that $\mathbb{E}[\min\{t \, : \, X_{t} = n\}] \approx n$. This gives the upper bound
\[
\tau_{mix}(Q) \lesssim n.
\]
To get the matching lower bound, first note that any walk $X_{t} \sim Q$ must satisfy $|X_{t}-X_{t+1}|\leq 1$ for all $t$. Thus, for $t < \frac{n}{2}$,
\[
d_{TV}(Q^{t}(n,\cdot) , \mu) \geq \mu([0,\frac{n}{3}]).
\]
Applying Claim 3 completes the proof that
\[
\tau_{mix}(Q) \gtrsim n.
\]

Next, we analyze the mixing time of $K$. The lower bound is very similar to the lower bound on $Q$. Let $Y_{t} \sim K$ be a chain started at $Y_{0} = 0$ and let $\tau(n) = \min\{t \, : \, Y_{t} =n\}$. For $t < \frac{n}{2}$, we have
\begin{align*}
d_{TV}(K^{t}(0,\cdot) , \nu) &\geq \mathbb{P}[\tau(n) > t] \cdot \nu([\frac{2n}{3},n]) \\
&= (1-\epsilon)^{t} \cdot \nu([\frac{2n}{3},n]).
\end{align*}
Considering $t < \frac{n\log 4 }{C_n}$ and applying Claim 4 to the second term completes the proof that
\[
\tau_{mix}(K) \gtrsim \frac{n}{C_n}.
\]
\end{proof}

\textbf{Claim 3:} The stationary distribution $\mu$ of $Q$ assigns probability $\mu([0,\frac{n}{3}]) \rightarrow 1$ as $n \rightarrow \infty$.

\begin{proof}
We recognize that $Q$ is obtained by taking a simple birth-and-death chain and inserting a holding probability of $\frac{1}{2}$ at each step. Thus, we can use the standard exact formula for the stationary measure of a birth-and-death chain to check this.
\end{proof}

\textbf{Claim 4:} The stationary distribution $\nu$ of $K$ assigns probability $\nu([\frac{2n}{3},n]) \rightarrow 1$ as $n \rightarrow \infty$.

\begin{proof}
Let $X_{t}$ be a Markov chain drawn from $K$, and let $\tau_{k+1} = \min\{t > \tau_{k} : X_{t} = n\}$ represent the successive times at which the chain hits $n$. We assume $X_{0} = n$ and set $\tau_{0}=0$.

Define $\delta_{j} = \tau_{j} - \tau_{j-1}$ as the times between successive visits to $n$. We note that $\delta_j$, $j\geq 1$ are independently and identically distributed.
Let $S_{k} = |\{ \tau_{k} \leq t < \tau_{k+1} \, : \, X_{t} < \frac{2n}{3}\}|$ be the amount of time between $\tau_{k}$ and $\tau_{k+1}$ that the chain spends below $\frac{2n}{3}$. By the strong law of large number,
\[
\nu([0,\frac{2n}{3})) = \lim_{K \rightarrow \infty} \frac{\sum_{k=1}^{K} S_{k}}{\sum_{k=1}^{K} \delta_{k}} = \frac{\lim_{K \rightarrow \infty} \frac{1}{K} \sum_{k=1}^{K} S_{k}}{\lim_{K \rightarrow \infty} \frac{1}{K} \sum_{k=1}^{K} \delta_{k}} = \frac{\mathbb{E}[S_{1}]}{\mathbb{E}[\delta_{1}]}
\]
holds almost surely.
On the one hand, since there is a probability $\epsilon$ of teleportation at each time step, we have $\mathbb{E}[S_1]\leq \frac{1}{\epsilon}$.
On the other hand, when hitting the state below $\frac{2n}{3}$, the chain must have spent at least $\frac{n}{3}$ time steps above $\frac{2n}{3}$. This leads to $\mathbb{E}[\delta_1]\geq \frac{n}{3}$.
Putting this together, we have
\[
\nu([0,\frac{2n}{3})) = \frac{\mathbb{E}[S_{1}]}{\mathbb{E}[\delta_{1}]} \leq  \frac{1/\epsilon}{n/3} =\frac{3}{C_n},
\]
which goes to 0 as long as $C_n \rightarrow \infty$.

\end{proof}

\subsection{Short Calculation on TV Distances}

\begin{lemma}\label{calculation-tv}
Fix $0 < r < 1$ and  let $\mu, \nu$ be two distributions on $\{-1,+1\}$ satisfying
\[
r \leq \frac{\mu(x)}{\nu(x)} \leq r^{-1}
\]
for $x \in \{-1,1\}$. Then
\[
d_{TV}(\mu,\nu) \leq \frac{1-r}{1+r}.
\]
\end{lemma}

\begin{proof}
By symmetry, the worst-case pair $\mu,\nu$ must satisfy
\[
\mu(+1) = \nu(-1) = \delta, \quad \mu(-1) = \nu(+1) = 1- \delta
\]
for some $0 < \delta < \frac{1}{2}$ satisfying
\[
\frac{\delta}{1-\delta} = r.
\]
Solving this for $\delta$, we have
\[
\delta = \frac{r}{1+r}.
\]
Then the total variation distance is:
\[
d_{TV}(\mu,\nu) = \nu(+1) - \mu(+1) = 1 - 2 \delta = \frac{1-r}{1+r}.
\]
\end{proof}

\end{document}